\documentclass[reqno,tbtags,intlimits,a4paper,oneside,11pt]{amsart}
\usepackage{amsfonts}
\usepackage{}
\usepackage{amsfonts,amsmath,amssymb}
\newtheorem{Theorem}{Theorem}[section]
\newtheorem{prop}[Theorem]{Proposition}
\newtheorem{Lemma}[Theorem]{Lemma}

\newtheorem{Corollary}[Theorem]{Corollary}

\def\beq#1#2\eeq{%
        \begin{equation}%
        \label{#1}%
            #2%
        \end{equation}%
    }

\usepackage{color}
\usepackage{graphicx}
\usepackage{epstopdf}

\title[Markov fractions and slopes]{Markov fractions and the slopes of the exceptional bundles on $\mathbb P^2$}

\author{A.P. Veselov}
\address{Department of Mathematical Sciences,
Loughborough University, Loughborough LE11 3TU, UK}
\email{A.P.Veselov@lboro.ac.uk}

\begin{document}

\maketitle

\begin{abstract}
We show that the Markov fractions introduced recently by Boris Springborn are precisely the slopes of the exceptional vector bundles on $\mathbb P^2$ studied in 1980s by Dr\`ezet and Le Potier and by Rudakov. In particular, we provide a simpler proof of Rudakov's result claiming that the ranks of the exceptional bundles on $\mathbb P^2$ are Markov numbers.

\end{abstract}

\section{Introduction}

In 1980s Dr\`ezet and Le Potier \cite{DLP} and Rudakov \cite{Rud} studied the {\it exceptional vector bundles} $E$ on $\mathbb P^2$ which are the stable vector bundles, which are rigid in the sense that $Ext^1(E,E)=0.$ In particular, it was proved that the rank $r(E)$ of such bundle must be a Markov number and the corresponding {\it slope} $\mu(E)$ defined as
$$
\mu(E)=\frac{c_1(E)}{r(E)},
$$ 
where $c_1(E) \in H^2(\mathbb P^2) \cong \mathbb Z$ is the first Chern class of $E$, determines such bundle uniquely.
Recall that the {\it Markov numbers} 
$$
1, 2, 5, 13, 29, 34, 89, 169, 194, 233, 433, 610, 985, \dots
$$
are parts of {\it Markov triples}, which are the positive integer solutions of the Markov equation
$$x^2 + y^2 + z^2=3xyz.$$ Markov \cite{Markov} showed that all such solutions can be found recursively from $(1,1,1)$ by Vieta involutions and permutations.%using a natural action of the modular group $PSL_2(\mathbb R)$ on the solutions of this equation. 

Markov triples were introduced by Markov in 1880 in relation with the number theory of the binary quadratic forms (see \cite{Cassels}), but in the last decades they became of significance in other parts of mathematics, including the theory of Frobenius manifolds and algebraic geometry \cite{D1,HP}.

Recently Boris Springborn \cite{Springborn} introduced the notion of the {\it Markov fractions} $\frac{p}{q}$, which are also parts of special triples of rationals with the denominators being Markov numbers:
$$
\frac{0}{1}, \frac{1}{2}, \frac{2}{5}, \frac{5}{13}, \frac{12}{29}, \frac{13}{34}, \frac{34}{89}, \frac{70}{169}, \frac{75}{194}, \frac{89}{233}, \frac{179}{433}, \frac{233}{610},  \frac{408}{985},  \dots 
$$
He proved that these fractions have remarkable Diophantine properties, being together with their companions the worst approximable rational numbers.
He also reveals deep connection with the hyperbolic geometry known for the Markov numbers since the work of Gorshkov \cite{Gorshkov} and Cohn \cite{Cohn}.

The main result of this paper is that the Markov fractions are precisely all possible slopes of the exceptional vector bundles on $\mathbb P^2.$
The proof can essentially be extracted from Rudakov \cite{Rud}, but we derive it directly from the results of Dr\`ezet and Le Potier \cite{DLP}, who described all the exceptional slopes as the values of certain function $\epsilon(x)$ defined on the dyadic rationals.
In particular, we provide a simpler proof of Rudakov's result claiming that the ranks of the exceptional bundles on $\mathbb P^2$ are Markov numbers.
%Finally we discuss the extension of the Dr\`ezet-Le Potier function and the related Springborn function to the real
%In the next section we will introduce the Markov fractions following \cite{Springborn} and study their properties. which will be used in the proof.

\section{Markov fractions}

In this section we mainly follow Springborn \cite{Springborn}, who proposed two equivalent definitions of Markov fractions.

We choose the one using the modification of the {\it Farey tree}, which is the projective version of the {\it Conway topograph} introduced by John H. Conway \cite{Conway}. 
The Conway topograph consists of the planar domains which are connected components of the complement to the trivalent  tree imbedded in the plane. These domains were originally labelled by the superbases in the integer lattice $\mathbb Z^2$, but can be equivalently labelled by the rationals using Farey mediant $\frac{a_1}{b_1}\oplus \frac{a_2}{b_2}=\frac{a_1+a_2}{b_1+b_2}$. We will be interested only in rationals $a/b \in [0,1]$ occupying part of the Conway topograph shown on Fig. 1.

\begin{figure}[h]
\begin{center}
\includegraphics[width=92mm]{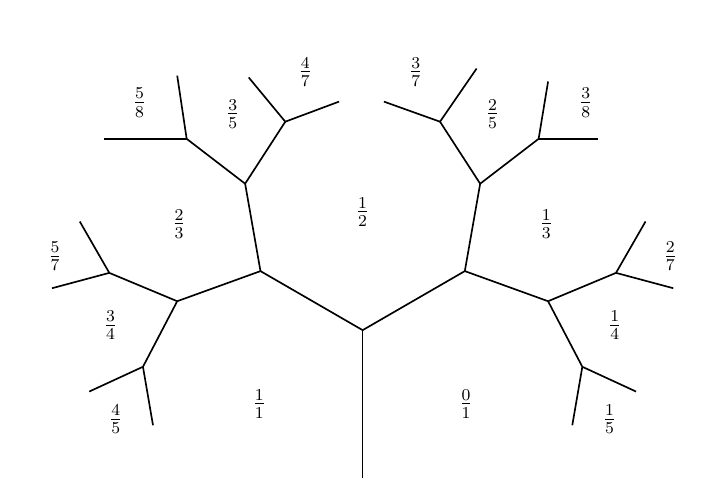} \, \includegraphics[width=30mm]{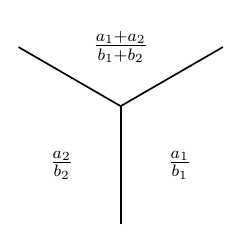}\,
\caption{\small The Farey tree of fractions between 0 and 1.}
\end{center}
\end{figure}

The {\it Markov fraction tree} is the modification of this Farey tree, where the Farey mediant is replaced by the {\it Springborn mediant} 
\beq{Sm}
\frac{p_1}{q_1}*\frac{p_2}{q_2}=\frac{p_1q_1+p_2q_2}{q_1^2+q_2^2},\eeq 
or, in the reduced form,
\beq{mfrac}
\frac{p_1}{q_1}*\frac{p_2}{q_2}=\frac{p}{q}, \quad p=\frac{p_1q_1+p_2q_2}{p_2q_1-p_1q_2}, \,\,\, q=\frac {q_1^2+q_2^2}{p_2q_1-p_1q_2}.
\eeq

\begin{figure}[h]
\begin{center}
\includegraphics[width=92mm]{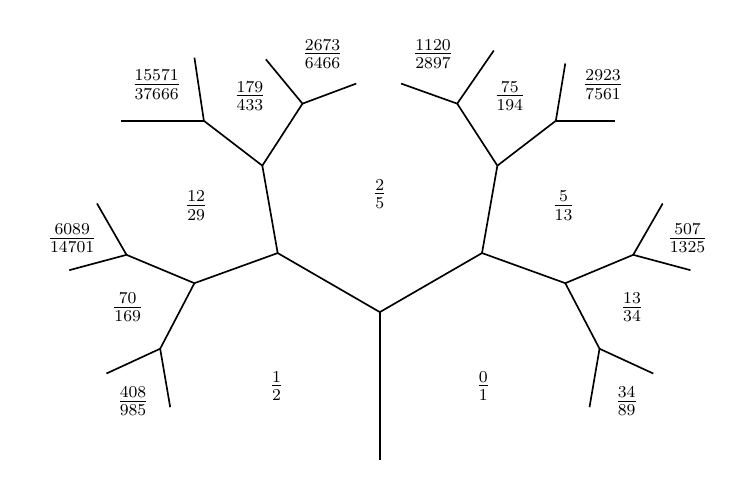}\, \includegraphics[width=30mm]{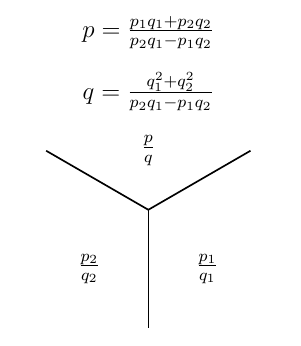}
\caption{\small The Markov fraction tree with the Springborn local rule.}
\end{center}
\end{figure}

By definition, the {\it Markov fractions} between $0$ and $1/2$ are defined recursively using the rule (\ref{mfrac}) starting from the fractions $\frac{0}{1}$ and $\frac{1}{2}$ (see Fig.2).
The set of such fractions we denote as $\mathcal {MF}_R$.

Juxtaposition with the Farey tree establishes the bijection \cite{Springborn}
\beq{bij}
\mu: \mathbb Q\cap [0,1] \to \mathcal {MF}_R,
\eeq 
which is a version of Frobenius parametrisation of Markov numbers \cite{Frobenius}.
By definition the function $\mu(x)$ satisfies the property
\beq{muprop}
\mu\left(\frac{a}{b}\oplus\frac{c}{d}\right)=\mu\left(\frac{a}{b}\right)*\mu\left(\frac{c}{d}\right),\quad |ad-bc|=1,
\eeq
intertwining Farey and Sprinborn mediants of neighbours on the Farey tree.
%where $\mathcal {MF}_R$ is the set of Markov fractions between $0$ and $1/2.$ 
The set of all Markov fractions is defined as 
\beq{MF}
\mathcal {MF}:=\{n\pm \frac{p}{q}, \quad  \frac{p}{q}\in \mathcal {MF}_R,\,\,  n \in \mathbb Z\}. 
\eeq
The reduced set $\mathcal {MF}_R=\mathcal {MF}\cap [0,1/2]$ is a fundamental domain of the natural action on $\mathcal {MF}$ of the integer affine group $\mathit{Aff}_1(\mathbb Z)$.

%In Fig.2 we see  the branch of the Markov fraction tree containing all Markov fractions in $[0,1/2]$, which is the fundamental domain $\mathcal MF$ , acting on the set of all Markov fractions by $x \to \pm x + n, \, n \in \mathbb Z$.
%so we can restrict ourselves with the set $\mathcal MF$ of all Markov fractions $0\leq \frac{p}{q}\leq \frac{1}{2}$ between 0 and $1/2$.

Springborn proved that the Markov fractions $\frac{p}{q}$ together with their properly defined companions are the worst approximable rationals having the corresponding {\it approximation constants} $C\left(\frac{p}{q}\right) \geq \frac{1}{3}$, where
$$
C\left(\frac{p}{q}\right):=\inf_{\frac{a}{b} \in \mathbb Q \setminus \bigl\{\frac{p}{q} \bigr\}} b^2 \left|\frac{p}{q}-\frac{a}{b}\right|
$$
(in a different form some related results were found earlier by Gbur \cite{Gbur}).

Note that Markov \cite{Markov} studied the quadratic forms corresponding to the irrational numbers $\alpha$ with the {\it Lagrange number} $L(\alpha)<3$, 
where $$
L(\alpha)^{-1}:=\liminf\limits_{b \rightarrow \infty} \left(b^2 \min_{a \in \mathbb Z} \left|\alpha-\frac{a}{b}\right|\right).
$$
(see  \cite{Aigner} and \cite{Cassels} for the details). %and Malyshev \cite{Malyshev} for more details).

Springborn also proved that Markov fractions $x=\frac{p}{q}$ can be characterised geometrically by the property that the vertical geodesic joining $x$ and infinity on the upper half-plane $\mathcal H^2$ projects to a simple geodesic on the modular torus with both ends in the cusp (see Corollary 4.3 in \cite{Springborn}). 

We start with the proof that, in contrast to the Farey mediants, the Springborn mediants (\ref{Sm}) are highly reducible, and that their reduced form is indeed given by formula (\ref{mfrac}).

\begin{prop}
The numbers $p,q$, which are defined recursively on the Conway topograph by the Springborn local rule (\ref{mfrac}) with the initial data $\frac{0}{1}$ and $\frac{1}{2},$ are coprime integers. 
The denominators $q$ of Markov fractions are the corresponding Markov numbers.
\end{prop}

\begin{proof}
As we have already mentioned, Markov \cite{Markov} proved that all the positive integer solutions of the Markov equation
\beq{Markov}
q_1^2+q_2^2+q_3^2=3q_1q_2q_3
\eeq
can be found from the obvious solution $(1,1,1)$ by applying permutations and Vieta involution $(q_1,q_2,q_3)\to (q_1,q_2,q_3')$ where
\beq{Vieta}
q_3'=3q_1q_2-q_3=\frac{q_1^2+q_2^2}{q_3}.
\eeq
This can be used to compute all Markov numbers recursively on the Conway topograph \cite{BV, Fock} (see Fig. 3).

\begin{figure}[h]
\begin{center}
\includegraphics[width=77mm]{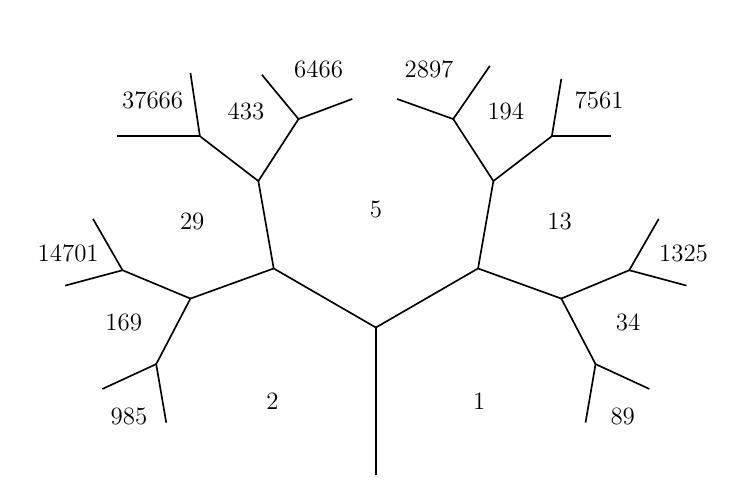}\, \includegraphics[width=40mm]{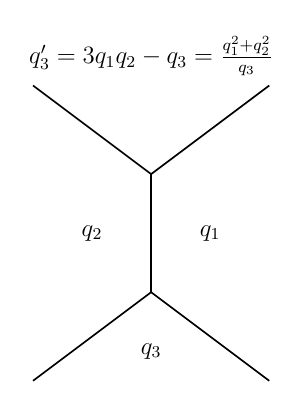}
\caption{\small Markov numbers on the Conway topograph.}
\end{center}
\end{figure}

The numbers in three domains meeting at one vertex form Markov triples (the singular Markov triples (1,1,1) and (1,1,2) are left outside the chosen part of the full Conway topograph).

Consider a part of the Markov fraction tree shown on Fig. 4.

\begin{figure}[h]
\begin{center}
\includegraphics[width=60mm]{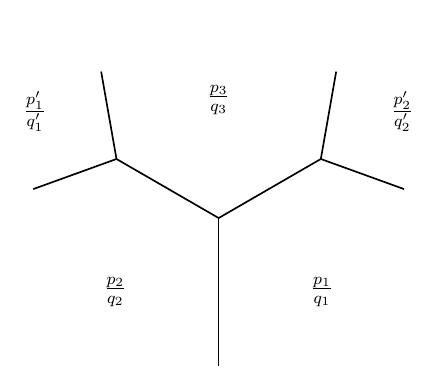} %\, \includegraphics[width=35mm]{Markov_law-0}
\caption{}
\end{center}
\end{figure}

\begin{Lemma}
The numbers on a part of the Markov fraction tree shown on Fig. 4 satisfy the following relations:
\beq{rel12}
p_2q_3-p_3q_2=q_1,\quad
p_3q_1-p_1q_3=q_2,
\eeq
\beq{rel3}
p_2q_1-p_1q_2=\frac{q_1^2+q_2^2}{q_3}=3q_1q_2-q_3,
\eeq
\beq{rel1'}
p_1'=\frac{p_2q_2+p_3q_3}{q_1}, \quad q_1'=\frac{q_2^2+q_3^2}{q_1},
\eeq
\beq{rel2'}
p_2'=\frac{p_1q_1+p_3q_3}{q_2}, \quad q_2'=\frac{q_1^2+q_3^2}{q_2}.
\eeq
\end{Lemma}

\begin{proof} The relation (\ref{rel12}) follows from the direct calculation:
$$
p_2q_3-p_3q_2=\frac{p_2(q_1^2+q_2^2)}{p_2q_1-p_1q_2}-\frac{q_2(p_1q_1+p_2q_2)}{p_2q_1-p_1q_2}=\frac{p_2q_1^2-p_1q_1q_2}{p_2q_1-p_1q_2}=q_1.
$$
The last two relations are now obvious:
$$
p_1'=\frac{p_2q_2+p_3q_3}{p_2q_3-p_3q_2}=\frac{p_2q_2+p_3q_3}{q_1}, \quad q_1'=q_1'=\frac{q_2^2+q_3^2}{p_2q_3-p_3q_2}=\frac{q_2^2+q_3^2}{q_1}
$$
(and similarly for (\ref{rel2'})). Comparing (\ref{rel1'}) with (\ref{Vieta}) we see that the denominators of Markov fractions are Markov numbers.
The relation (\ref{rel3}) now follows from the construction of the tree. This proves the lemma.
\end{proof}

%Comparing (\ref{rel1'}) with (\ref{Vieta}) we see that the denominators of Markov fractions are Markov numbers. 

Note that the integrality of $q_1'$ follows from the second Vieta relation for the Markov equation:
$$
q_1'=\frac{q_2^2+q_3^2}{q_1}=3q_2q_3-q_1.
$$
From this relation it follows also that all Markov numbers from the same Markov triple are pairwise coprime since this is true for the seed $(1,1,1).$

To prove the integrality of $p_1'$ consider the product
$$
q_2(p_2q_2+p_3q_3)=p_2q_2^2+p_3q_2q_3=p_2q_2^2+(p_2q_3-q_1)q_3=p_2(q_2^2+q_3^2)-q_1q_3
$$
$$
=p_2q_1q_1'-q_1q_3=q_1(p_2q_1'-q_3).
$$
Thus $q_1$ divides $q_2(p_2q_2+p_3q_3)$ and since $q_1$ and $q_2$ are coprime, $q_1$ divides $p_2q_2+p_3q_3$, which implies the integrality of $p_1'.$

The fact that $p_3$ and $q_3$ are coprime follows from the relations (\ref{rel12}) since $q_1$ and $q_2$ are coprime.
\end{proof}

\begin{prop}
For any Markov fraction $\frac{p}{q}$ the numerator $p$ satisfies the quadratic congruence
\beq{qcong}
p^2+1\equiv 0 \,(mod \,\, q).
\eeq
\end{prop}

\begin{proof}
To prove this for $p=p_1', q=q_1'$ consider
$$
(p_1')^2+1=\frac{(p_2q_2+p_3q_3)^2}{(p_2q_3-p_3q_2)^2}+1=\frac{(p_2^2+p_3^2)(q_2^2+q_3^2)}{(p_2q_3-p_3q_2)^2}=\frac{(p_2^2+p_3^2)(q_2^2+q_3^2)}{q_1^2}.
$$
Note that
$$
(p_2^2+p_3^2)(q_2^2+q_3^2)=(p_2^2+p_3^2)q_1q_1' \equiv 0 \,(mod \,\, q_1').
$$
Since $q_1$ and $q_1'$ are coprime, we can divide this congruence by $q_1^2$ to get the congruence (\ref{qcong}).
\end{proof}

Consider now some special branches of Markov fraction tree. 

The bottom right branch
$$
\frac{1}{2},\, \frac{2}{5}, \,\frac{5}{13}, \,\frac{13}{34}, \, \frac{34}{89}, \,  \frac{89}{233}\dots
$$
consists of the fractions $\frac{p_k}{q_k}$ by construction satisfying the recurrence
$$
p_{k+1}=\frac{p_kq_k+0}{p_k}=q_k, \quad q_{k+1}=\frac{q_k^2+1}{p_k}=\frac{q_k^2+1}{q_{k-1}}.
$$
The corresponding Markov numbers $q_k=F_{2k}$ are known to be every second Fibonacci number (see e.g. \cite{Aigner})
$$
F_k={\color{blue}1}, 1, {\color{blue}2}, 3, {\color{blue}5}, 8, {\color{blue}13}, 21, {\color{blue}34}, 55, {\color{blue}89}, 144, {\color{blue}233}, \dots,
$$ 
so the corresponding Markov fractions are $F_{2k}/F_{2k+2}.$

To describe the bottom left branch 
$$
 \frac{2}{5}, \,\frac{12}{29}, \,\frac{70}{169}, \, \frac{408}{985}, \,  \frac{89}{233}\dots
$$
define {\it Pell numbers} $(x_n, y_n)$ as the positive solutions of Pell's equations
$$x^2-2y^2=(-1)^n:$$
$$
(1,1), (3,2), (7,5), 17,12), (41,29), (99,70), (239,169), (577, 408), (1393,985), \dots
$$
From the theory of Pell equation (see e.g. \cite{LeVeque}) these numbers can be found recursively:
$$
x_{n+1}=2x_n+x_{n-1}, \quad x_1=1, x_2=3
$$
$$
y_{n+1}=2y_n+y_{n-1}, \quad y_1=1, y_2=2.
$$

\begin{prop}
Bottom left branch of Markov fraction tree consists of the fractions
\beq{pell}
\frac{p_k}{q_k}=\frac{y_{2k}}{y_{2k+1}},
\eeq
where $y_n$ are the Pell numbers.
\end{prop}

\begin{proof}
The corresponding Markov numbers $q_k$ are known to be the Pell numbers $y_{2k+1}$ \cite{Aigner}, so we need only to prove that $p_k=y_{2k}$.
We prove this by induction. This is clearly true for $k=1.$ Assume now that the relation (\ref{pell}) holds for $k\leq n$, so that
$$
p_{k+1}=\frac{p_kq_k+2}{q_{k-1}}=\frac{y_{2k}y_{2k+1}+2}{y_{2k-1}}.
$$
To prove that $\frac{y_{2k}y_{2k+1}+2}{y_{2k-1}}=y_{2k+2}$ we need to show that $W_{2k}=2$, where
$
W_j:=y_{j+2}y_{j-1}-y_j y_{j+1}.
$
We have
$$
W_j+W_{j-1}=y_{j+2}y_{j-1}-y_j y_{j+1}+y_{j+1}y_{j-2}-y_{j-1} y_{j}
$$
$$=y_{j-1}(y_{j+2}-y_j)-y_{j+1}(y_{j}-y_{j-2})=2y_{j-1}y_{j+1}-2y_{j+1}y_{j-1}=0,
$$
This means that $W_j=-W_{j-1}$, so $W_{2k}=W_{2k-2}$ does not depend on $k$ and from the initial data we have $W_{2k}= 2$.
\end{proof}

\section{Exceptional slopes and Markov fractions}

For the general theory of the algebraic vector bundles on $\mathbb P^2$ we refer to Le Potier \cite{LePotier}.

Let $E$ be an algebraic vector bundle on complex projective plane $\mathbb P^2$ of rank $r$ with the Chern classes $c_1$ and $c_2.$ 
Since $H^2(\mathbb P^2)\cong \mathbb Z$ and $H^4(\mathbb P^2)\cong \mathbb Z$ we can consider $c_1$ and $c_2$ as integers.
The ratio 
$$
\mu(E):=\frac{c_1}{r}
$$
is called {\it slope}. 
The algebraic vector bundle $E$ (which can be considered as a locally-free sheaf) is called {\it stable} if for any proper sub-sheaf $F \subset E$ we have
$$
\mu(F)<\mu(E)
$$
and {\it rigid} if $Ext^1(E,E)=0.$ The vector bundles, which are both stable and rigid, are called {\it exceptional.} Alternatively, the exceptional bundles can be defined by the conditions
$$
Hom(E,E) = \mathbb{C},  \,\,  Ext^i(E,E) = 0,  i > 0.
$$

It is known \cite{DLP} that the exceptional vector bundles $E$ on $\mathbb P^2$ are uniquely determines by the slope $\mu(E)$,
so the main question is to describe the set $\mathfrak E$ of fractions $\frac{p}{q}$ for which there exists an exceptional bundle $E$ with the slope $\mu(E)=\frac{p}{q}.$ 

The main result of this paper is the following

\begin{Theorem}
The set $\mathfrak E$ of slopes of the exceptional bundles on $\mathbb P^2$ coincides with the set of all Markov fractions.
\end{Theorem}

\begin{proof}

One can prove this using the results of Rudakov \cite{Rud}, but we derive it directly from the results of Dr\`ezet and Le Potier \cite{DLP}.

Since the set of exceptional bundles is invariant under tensor product by $\mathcal O(n)$ and taking the dual, the corresponding set of possible slopes $\mathfrak E$ is invariant under the integer affine group $Aff_1(\mathbb Z)$. This means that it is enough to describe the intersection $\mathfrak E \cap [0,1/2]$, which we claim to be the set $\mathcal MF$ of Markov fractions from $[0,1/2]$ (see Fig. 1).

Dr\`ezet and Le Potier \cite{DLP} described the set $\mathfrak E$ as the image of a special function $\epsilon: \mathfrak D \to \mathbb Q$, where $\mathfrak D$ is the set of dyadic (binary) rationals $m/2^n$.
This function has the properties $$ \epsilon(-x)=-\epsilon(x), \quad \epsilon(x+n)=\epsilon(x)+n, \, n\in \mathbb Z$$ and is uniquely defined by the condition that, if $\epsilon\left(\frac{m}{2^{n}}\right)=\frac{p_1}{q_1}, \,\,\, \epsilon\left(\frac{m+1}{2^{n}}\right)=\frac{p_2}{q_2},$ then
\beq{epsil}
\frac{p_3}{q_3}:=\epsilon\left(\frac{2m+1}{2^{n+1}}\right)=\frac{1}{2}\left(\frac{p_1}{q_1}+ \frac{p_2}{q_2} +\frac{q_1^{-2}-q_2^{-2}}{p_1/q_1-p_2/q_2+3}\right).
\eeq

The key observation now is that the Dr\`ezet-Le Potier defining relation (\ref{epsil}) is equivalent to the Springborn mediant rule (\ref{Sm}).
Indeed, we claim that for the neighbouring Markov fractions $p_1/q_1<p_2/q_2$ we have the identity
\beq{iden}
\frac{1}{2}\left(\frac{p_1}{q_1}+ \frac{p_2}{q_2} +\frac{q_1^{-2}-q_2^{-2}}{p_1/q_1-p_2/q_2+3}\right)=\frac{p_1q_1+p_2q_2}{q_1^2+q_2^2}.
\eeq
%where we assume that $p_1/q_1<p_2/q_2$. 
First let us simplify
$$
\frac{q_1^{-2}-q_2^{-2}}{p_1/q_1-p_2/q_2+3}=\frac{q_2^2-q_1^2}{(3q_1q_2+p_1q_2-p_2q_1)q_1q_2}=\frac{q_2^2-q_1^2}{q_1q_2q_3},
$$
where we have used relation (\ref{rel3}) from Lemma 2.2. So we need to show that
$$
\frac{(p_1q_2+p_2q_1)q_3+q_2^2-q_1^2}{2q_1q_2q_3}=\frac{p_1q_1+p_2q_2}{q_1^2+q_2^2},
$$
or, equivalently, that
$$
[(p_1q_2+p_2q_1)q_3+q_2^2-q_1^2](q_1^2+q_2^2)-2q_1q_2q_3(p_1q_1+p_2q_2)=0.
$$
After simple algebra the right hand side takes the form
$$
(q_2^2-q_1^2)[(p_1q_2-p_2q_1)q_3+q_1^2+q_2^2]=(q_2^2-q_1^2)[(q_3-3q_1q_2)q_3+q_1^2+q_2^2]
$$
$$=(q_2^2-q_1^2)(q_1^2+q_2^2+q_3^2-3q_1q_2q_3)=0
$$
due to Markov equation and relation (\ref{rel3}). This proves the identity (\ref{iden}).

Thus we see that the description of the exceptional slope set by Dr\`ezet and Le Potier is equivalent to the construction of the Markov fraction tree by the Springborn rule.
Now the theorem follows from the results of \cite{DLP}.
\end{proof}

Note that as a result we have a simpler proof of the well-known result by Rudakov \cite{Rud}.

\begin{Corollary}
The ranks of the exceptional vector bundles on $\mathbb P^2$ are Markov numbers.
\end{Corollary}

\section{Extension to real numbers and Minkowski question mark function}

We have two functions on $\mathbb Q\cap [0,1]$ with the same value set $\mathcal{MF}$ : the Dr\`ezet-Le Potier function $\epsilon(x)$ and the Springborn function $\mu(x).$
They have the properties which can be written in terms of Farey and Sprinborn mediants as
%In the notations from page 2 we can rewrite now the Drezet-Le Potier property (\ref{epsil}) in terms of Farey and Sprinborn mediants as
\beq{epsil1}
\epsilon\left(\frac{m}{2^{n}}\oplus\frac{m+1}{2^{n}}\right)=\epsilon\left(\frac{m}{2^{n}}\right)*\epsilon\left(\frac{m+1}{2^{n}}\right),
\eeq
\beq{mupro}
\mu\left(\frac{a}{b}\oplus\frac{c}{d}\right)=\mu\left(\frac{a}{b}\right)*\mu\left(\frac{c}{d}\right), \quad |ad-bc|=1.
\eeq
A natural question is to consider and study the extension of these monotonically increasing functions to real $x \in [0,1].$ %The case of $\mu(x)$ was already discussed by Springborn in \cite{Springborn} (see Section 2.4), but most of the questions remain open.

It is natural to compare these functions with the {\it question mark function} $?(x)$ introduced by Minkowski \cite{Mink} and studied later by Denjoy and by Salem (see more details in \cite{ViaParadis}). It
can be uniquely defined by the properties: $?(0)=0, \,\, ?(1)=1$ and, if $\frac{a}{b}$ and $\frac{c}{d}$ are neighbours on the Farey tree, then
\beq{deffar}
?\left(\frac{a}{b}\oplus\frac{c}{d}\right)=\frac{1}{2}\left(?\left(\frac{a}{b}\right)+?\left(\frac{c}{d}\right)\right), \quad |ad-bc|=1
\eeq
(so the corresponding mediant is just the aritmetic mean).

Salem  \cite{Salem} gave an equivalent definition of $?(x)$ in terms of the continued fraction expansion of $x = [0, a_1, a_2, \dots, a_n, \dots]:$
\begin{equation}
\label{Salem}
?(x) = \frac{1}{2^{a_1 - 1}} - \frac{1}{2^{a_1 + a_2 - 1}} + \frac{1}{2^{a_1 + a_2 + a_3 - 1}} - \dots .
\end{equation}
Spalding and the author \cite{SV1} used this to describe this function on the Conway topograph. 
Namely, if $\gamma_x$ is a path on the Farey tree leading to $x \in [0,1]$ (for irrational, infinite), then
$$
?(x)= [0. a_{1} a_{2} \dots a_{j} \dots]_2, 
$$
where the binary digits $a_j$ equal to $0$ or $1$ depending whether the $j$-th step of $\gamma_x$ is a right turn, or left turn (see Section 6 in \cite{SV1}).

 The Minkowski question mark function can be extended continuously to $[0,1],$ where it has  the following properties (see \cite{ViaParadis}):
\begin{itemize}
\item $?(x)$ has finite binary representation (dyadic rational) iff $x$ is rational
\item $x$ is a quadratic irrational iff $?(x)$ is rational, but not dyadic rational
\item $?(x)$ is strictly increasing and defines a homeomorphism of $[0,1]$ to itself
\item The derivative $?'(x) = 0$ almost everywhere.
\end{itemize}

In contrast to the Minkowski function, the extension of both functions $\epsilon(x)$ and $\mu(x)$ are discontinuous at all rational $x.$ Indeed, Springborn proved that the interval around any Markov fraction $\frac{p}{q}$
\beq{int}
I=\left[\frac{p}{q}-\frac{1}{2}l(q), \, \frac{p}{q}+\frac{1}{2}l(q),\right],\quad l(q)=3-\frac{\sqrt{9q^2-4}}{q}
\eeq
 is the maximal interval free of other Markov fractions, 
 so at each rational $x$ with $\mu(x)=\frac{p}{q}$ the function $\mu$  has jump $l(q)$.
For example, for $x=1/2$ the interval $I=\left[\frac{-11+\sqrt{221}}{10}, \frac{19-\sqrt{221}}{10}\right]$ contains the only Markov fraction $2/5$ and $\mu(x)$ has at $x=1/2$ the jump $(15-\sqrt{221})/10$.

Note that if we know all the jumps of a monotonic function $f(x)$ then we can define the so-called {\it saltus function} $s_f(x)$ such that $f(x)=s_f(x)+g(x)$ with a monotonic continuous $g(x)$ (see \cite{RS}, p. 14-15). Let $H(x)$ be the {\it Heaviside step function}:  
$$
    H(x) =
\begin{cases}
0, \,\, x<0,\\
1/2, \,\, x=0,\\
1, \,\, x>0.
\end{cases} 
$$

\begin{prop}
Springborn function coincides with its saltus function:
\beq{saltus}
\mu(x)=s_\mu(x):=-\frac{1}{2}l(1)+\sum_{a/b \in \mathbb Q\cap [0,1]}l\left(q\left(\frac{a}{b}\right)\right)H\left(x-\frac{a}{b}\right), 
\eeq
where $q(\frac{a}{b})$ is the Markov number, corresponding to $a/b$.
\end{prop}

The proof follows from the {\it McShane identity} \cite{McShane}, which can be written as \footnote{I am grateful to Boris Springborn for explaining this to me.}
\beq{mcs}
\frac{1}{2}(l(1)+l(2))+\sum_{q \in \mathcal M, \, q>2} l(q)=\frac{1}{2},
\eeq
where $\mathcal M$ is the set of all Markov numbers.

\begin{Corollary}
The derivative $\mu'(x)=0$ almost everywhere.
\end{Corollary}
%It is natural to ask about possible analogues of these properties for $\epsilon(x)$ and $\mu(x)$.

Springborn  \cite{Springborn} considered also the limits of Markov fractions along the infinite paths on the Markov fraction tree leading to the map
\beq{muhat}
\hat \mu:  \mathbb R_{\geq 0}\cup \{\infty\} \rightarrow \{\textit {$PGL_2(\mathbb Z)$-classes of real numbers}\}.
\eeq
For every rational $x$ there are exactly two infinite paths approaching $x$ with two different limits of the corresponding Markov fractions, which are two $PGL_2(\mathbb Z)$-equivalent  Markov irrationalities  $$\hat\mu(x)=\frac{p}{q}\pm \frac{1}{2}l(q)=\mu(x)\pm \frac{1}{2}l(q)$$ with the Lagrange number  $L(\hat\mu(x))=\frac{\sqrt{9q^2-4}}{q} <3. $
For irrational $x$ these limits are well-defined and have the Lagrange number $L=3.$ Whether this gives all such real numbers is an interesting open question \cite{Springborn}.

There is also a related {\it Lyapunov function} $\Lambda(x)$ describing the limits of  Markov numbers along the paths on the Conway topograph:
\beq{defL}
\Lambda(\xi)=\limsup_{n\to\infty}\frac{\ln(\ln q_n(\xi))}{n}, \,\, \xi \in \mathbb{R}P^1,
\eeq
where $q_n(\xi)$ is the corresponding Markov number along the path $\gamma_\xi$ leading to $\xi$.
It was studied in \cite{SV1}, where it was proved that $\Lambda(\xi)$ is $PGL_2(\mathbb Z)$-invariant function, which vanishes almost everywhere (but has the support with Hausdorff dimension equal to 1) and has the set of values $[0, \ln \varphi],$ where $\varphi=\frac{1+\sqrt 5}{2}$ is the golden ratio.
 
%\item The restriction of $\Lambda$ on the Markov-Hurwitz set $X$ of the most irrational numbers is monotonically increasing from $\Lambda(\sqrt{2})=\frac{1}{2} \ln (1+\sqrt{2})$ to $\Lambda(\varphi)=\ln \varphi$ and in the Farey parametrization is convex.

%This function  where it is strictly increasing and defines a homeomorphism of $?:[0,1]\to [0,1]$ such that $?'(x) = 0$ almost everywhere \cite{Salem}. For $x$ rational $?(x)$ is dyadic %rational, the rational values of $?(x)$, which are not dyadic rational, correspond quadratic irrationals $x.$ 

%It would be interesting to study similar questions for the extension to real $x$ of the Drezet-Le Potier function $\epsilon(x)$  
%and of the Springborn function $\mu(x)$ with property (\ref{muprop}).
%$$
%\mu\left(\frac{a}{b}\oplus\frac{c}{d}\right)=\mu\left(\frac{a}{b}\right)*\mu\left(\frac{c}{d}\right),
%$$
%where as before $\frac{a}{b}$ and $\frac{c}{d}$ are neighbours on the Farey tree. 

\section{Concluding remarks}

The celebrated {\it Unicity conjecture} by Frobenius \cite{Aigner} claims that any Markov triple is uniquely determined by its maximal part. Springborn reformulated it as the following 

{\bf Conjecture 1.} {\it For any Markov number $q$ there exists a unique Markov fraction $0\leq \frac{p}{q}\leq \frac{1}{2}.$}

This means that one needs to prove that every Markov fraction appears on the Markov fraction tree only once, which may help to settle this conjecture.

Since the exceptional slope $\mu(E)$ determines the exceptional bundle $E$ uniquely, we can reformulate it equivalently as follows. Consider the natural action of $Aff_1(\mathbb Z)$ on the vector bundles generated by the transformations $E\to E^*$ and $E\to E\otimes \mathcal O(n)$.

{\bf Conjecture 2.} {\it Every exceptional bundle $E$ on $\mathbb P^2$ is determined by its rank uniquely modulo natural action of $Aff_1(\mathbb Z)$.}

By Proposition 2.3 for an exceptional bundle $E$ with given rank $q=r(E)$ (which is a Markov number), the Chern class $p=c_1(E)$ satisfies the quadratic congruence
$$ 
p^2+1\equiv 0 \,(mod\,\, q).
$$
For prime Markov numbers $q$ this congruence has a unique solution (up to a sign), so both conjectures hold true in this case (which is already known due to Baragar \cite{Baragar}). Although conjecturally there are infinitely many prime Markov numbers, according to Bourgain, Gamburd and Sarnak \cite{BGS},
they have density zero among all Markov numbers, so the congruence in general have many solutions. 

For example, for the Markov fraction $\frac{15571}{37666}$ with $37666=2\times 37 \times 509$ the congruence
$
x^2+1\equiv 0 \,(mod\,\, 37666)
$
has 4 solutions $x\equiv \pm 2337, \, \pm 15571$.
It would be interesting to characterise the particular solution given by the numerator of the Markov fraction in some other terms.

Note that by the Riemann-Roch theorem the second Chern class $c_2(E)$ of the exceptional bundle $E$ can be computed in terms of $p=c_1(E), \, q=r(E)$ from the formula
\beq{s}
c_2(E)=\frac{1}{2}(q-1)(s+1),
\eeq
where the number $s$ is defined by
$p^2+1=sq.$

It is interesting that the {\it Markov binary quadratic form} \cite{Markov}, corresponding to Markov fraction $\frac{p}{q}\leq \frac{1}{2}$, can be written in these terms as
\beq{Mf}
f(x,y)=qx^2+(3q-2p)xy+(s-3p)y^2
\eeq
(see Definition 2.5 in \cite{Aigner}).

A natural question would be to understand the arithmetic of the exceptional bundles on the del Pezzo surfaces studied in  \cite{KN, KO, Rud2}.
In particular, it was shown by Rudakov \cite{Rud2}  that the exceptional collections on quadrics $\mathbb P^1\times \mathbb P^1$ have ranks $(x,y,z)$ satisfying the
Diophantine equation
$$
x^2+y^2+2z^2=4xyz
$$
with all positive solutions derived by mutations from
$(1,1,1).$

Karpov and Nogin \cite{KN} generalised this to other del Pezzo surfaces known to be isomorphic to either $\mathbb P^1\times \mathbb P^1$, or to $X_m$ being the plane $\mathbb P^2$ blown up in $m$ generic points with $0\leq m \leq 8.$ In particular, for $X_3$ we have the
Diophantine equation
$$
x^2+2y^2+3z^2=6xyz
$$
with all positive solutions derived by mutations from
$(1,1,1).$ 

A natural question is to describe the exceptional slopes in the del Pezzo cases. The work  \cite{Rosen} by Rosenberger might be helpful here.

Finally, I would like to mention the {\it Dubrovin conjecture} \cite{D2} claiming that for any Fano variety $X$ the derived category $D^b(X)$ has a full exceptional collection if and only if the quantum cohomology ring of $X$ is semisimple (see the details and latest development in \cite{MM}).

\section{Acknowledgements}

I am very grateful to Boris Springborn for attracting my attention to his remarkable paper \cite{Springborn} and to Alexander Kuznetsov, Dmitri Orlov and Artie Prendergast-Smith for very helpful discussions of the algebro-geometric aspects of the paper.

\end{document}